\documentclass[reqno]{amsart}
\usepackage{float}
\usepackage{amsmath}
\usepackage{graphicx}
\usepackage{latexsym}
\usepackage{amsfonts}
\usepackage{amssymb}
\theoremstyle{plain}
\newtheorem{theorem}{Theorem}
\newtheorem{corollary}[theorem]{Corollary}
\newtheorem{lemma}[theorem]{Lemma}
\newtheorem{proposition}[theorem]{Proposition}
\theoremstyle{definition}

\newtheorem{remark}[theorem]{Remark}
\newtheorem*{remark*}{Remark}

\newcommand{\pr}{\mathbf P}
\newcommand{\e}{\mathbf E}

\begin{document}
\title[Crossing of a curve boundary by a random walk]
{Exact asymptotics for the instant of crossing a curve boundary by an asymptotically
stable random walk}
\author[Denisov]{Denis Denisov} 
\address{School of Mathematics, University of Manchester, UK}
\email{denis.denisov@manchester.ac.uk}

\author[Wachtel]{Vitali Wachtel} \address{Mathematical Institute,
  University of Munich, Theresienstrasse 39, D--80333 Munich, Germany}
\email{wachtel@mathematik.uni-muenchen.de}

\begin{abstract}
Suppose that $\{S_n,\ n\geq0\}$ is an asymptotically stable random walk.
Let $g$ be a positive function and $T_g$ be the first time when $S_n$
leaves $[-g(n),\infty)$. In this paper we study asymptotic behaviour of
$T_g$. We provide integral tests for function $g$ that guarantee
$\pr(T_g>n)\sim V(g)\pr(T_0>n)$ where $T_0$ is the first strict descending
ladder epoch of $\{S_n\}$.
\end{abstract}

\maketitle

\section{Introduction and main results }

Consider a one-dimensional random walk  
$$
S_0=0,\ S_n=X_1+\cdots+X_n,\,n\geq1,
$$ 
where $X,X_1,X_2,\ldots$ are i.i.d. random variables. Let $g(t)$ be an
increasing function and consider the exit time from the domain bounded
by $-g(t)$, that is,
$$
T_g:=\min\{n\ge 1: S_n< -g(n)\}.
$$
The aim of this paper is to study the asymptotics $\pr(T_g>n)$
as $n$ goes to infinity.

This question has been thoroughly good studied in the case of constant boundary,
that is, $g(n)\equiv x\geq0$. In particular, it is well-known that if
\begin{align}
\pr(S_n\ge 0)\to \rho\in(0,1),\quad n\to\infty\label{eq1}
\end{align}
then
\begin{align}
\pr(T_x>n) \sim h(x)n^{\rho-1}L(n),\quad n\to \infty\label{eq2},
\end{align}
where $L(n)$ is a slowly varying function. (Slightly abusing  notation,
we write $T_x$ for $T_g$ in the case when $g(t)\equiv x$. We 
also write  $a(x)\sim b(x)$ if $a(x)/b(x)\to 1$ as $x$ to infinity.) 
Function $h(x)$ is the renewal function of the strict decreasing ladder
height process.

We now introduce a class of random walks, which will be considered in
the present paper. Let
\begin{equation*}
\mathcal{A}:=\{0<\alpha <1;\,|\beta |<1\}\cup \{1<\alpha <2;|\beta |\leq
1\}\cup \{\alpha =1,\beta =0\}\cup \{\alpha =2,\beta =0\}
\end{equation*}%
be a subset in $\mathbb{R}^{2}.$ For $(\alpha ,\beta )\in \mathcal{A}$ and a
random variable $X$ write $X\in \mathcal{D}\left( \alpha ,\beta \right) $ if
the distribution of $X$ belongs to the domain of attraction of a stable law
with characteristic function%
\begin{equation}
G_{\alpha ,\beta }\mathbb{(}t\mathbb{)}:=\exp \left\{ -c|t|^{\,\alpha
}\left( 1-i\beta \frac{t}{|t|}\tan \frac{\pi \alpha }{2}\right) \right\},
\ c>0,  \label{std}
\end{equation}%
and, in addition, $\mathbf{E}X=0$ if this moment exists. Let 
$c(x) $ be a positive function specified by the relation%
\begin{equation}
c(x):=\inf \left\{ u\geq 0:\mu (u)\leq x^{-1}\right\},\ x\geq1 ,  \label{Defa}
\end{equation}%
where
\begin{equation*}
\mu (u):=\frac{1}{u^{2}}\int_{-u}^{u}x^{2}\mathbf{P}(X\in dx).
\end{equation*}%
It is known (see, for instance, \cite[Ch. XVII, \S 5]{FE}) that for every
$X\in \mathcal{D}(\alpha ,\beta )$ the function $\mu (u)$ is regularly
varying with index $(-\alpha )$. This implies that $c(x)$ is regularly
varying with index $\alpha ^{-1}$, i.e., there exists a function $l_{1}(x)$, 
slowly varying at infinity, such that
\begin{equation}
c(x)=x^{1/\alpha }l_{1}(x).  \label{asyma}
\end{equation}%
In addition, the scaled sequence $\left\{ S_{n}/c(n),\,n\geq 1\right\} $
converges in distribution to the stable law given by (\ref{std}). In this
case we say that $S_n$ is an \emph{asymptotically stable random walk}.
For every $X\in \mathcal{D}\left( \alpha ,\beta \right)$
there is an explicit formula for $\rho$,
\begin{equation}
\displaystyle\rho =\left\{
\begin{array}{ll}
\frac{1}{2},\ \alpha =1, &  \\
\frac{1}{2}+\frac{1}{\pi \alpha }\arctan \left( \beta \tan \frac{\pi \alpha
}{2}\right) ,\text{ otherwise}. &
\end{array}%
\right.  \label{ro}
\end{equation}

Since $h(0)=1$, one can rewrite \eqref{eq2} in a slightly different way:
$$
\pr(T_x>n)\sim h(x)\pr(T_0>n).
$$
This representation shows that the asymptotics is the same up to a constant
for every fixed $x$. This statement remains valid for curved boundaries which
grow not very fast. For the Brownian motion and monotone increasing function
$g$ it was shown in \cite{Nov81} and \cite{U80} that 
\begin{equation}\label{eq3}
0<\e[|B(T^{(bm)}_g)|]<\infty \iff \int_1^\infty g(t)t^{-3/2}dt <\infty 
\end{equation}
and, moreover,
\begin{equation}\label{eq3a}
\pr(T^{(bm)}_g>t)\sim \e[B(T^{(bm)}_g)]\pr(T^{(bm)}_1>t)
\text{ as }t\to\infty.
\end{equation}
Finiteness of the integral in \eqref{eq3} is also necessary: if
$\int_1^\infty g(t)t^{-3/2}dt=\infty$ then 
$$
\pr\{T^{(bm)}_g>t\}\gg\pr\{T^{(bm)}_1>t\}\quad\text{as }t\to\infty.
$$

The case of general random walks and Levy processes was analysed by Greenwood
and Novikov in \cite{GN}, see also Novikov \cite{Nov81,Nov82}. In Theorem~1
they state that if $\e g(T_0)<\infty$  then there exists $R_g\in (0,\infty)$
such that
\begin{equation}
\label{GN_thm1}
\pr(T_g>n) \sim R_g \pr(T_0>n)\quad\text{as }n\to\infty.
\end{equation}
However, recent results by Aurzada and Kramm \cite{AK}  give strong
grounds to suspect that conditions in \cite{GN} are not optimal. More precisely,
it is shown in \cite{AK} that if $g(t)=o(t^\gamma)$ with some $\gamma<1/\alpha$
then $\mathbf{P}(T_g>n)$ and $\mathbf{P}(T_0>n)$ have the same rough
asymptotics, that is,
$$
\mathbf{P}(T_g>n)=n^{\rho-1+o(1)}\quad\text{ as }n\to\infty.
$$
But, in view of \eqref{eq2}, $\e g(\tau_0)$ is finite only for
$\gamma\leq 1-\rho<1/\alpha$.

In this paper we present alternative (and milder than $\mathbf{E}g(T_0)<\infty$)
conditions for the validity of \eqref{GN_thm1}.

\begin{theorem}\label{T1}
Let $g$ be an increasing function such that $h(g(x))$ is subadditive. 
If \eqref{eq1} holds and
\begin{equation}
\label{hg-moment}
\mathbf{E}h(g(T_0))<\infty
\end{equation}
then there exists a constant $V(g)\in (0,\infty)$ such that 
\begin{equation}
\label{T1.1}
\pr(T_g>n)\sim V(g)\pr (T_0>n) 
\sim V(g)n^{\rho-1}L(n),\quad n\to \infty.
\end{equation}
\end{theorem}
This result generalises Theorem 1 in Greenwood and Novikov \cite{GN}
where it was assumed that that $g$ is concave and that $\mathbf{E} g(T_0)$
is finite. Noting that $h$ is subadditive, we conclude that
\eqref{hg-moment} is weaker than the finiteness of $\mathbf{E}g(T_0)$.
Moreover, subadditivity of $h$ implies also that $h(g(x))$ is subadditive
for every concave function $g$. And the subadditivity assumption can be
further weakened:
\begin{remark}
The statement of Theorem \ref{T1} remains valid if we replace
subadditivity of $h(g(x))$ by the existence of a subadditive majorant
which satisfy \eqref{hg-moment}.
\hfill$\diamond$
\end{remark}
We next specialise Theorem \ref{T1} to the case of asymptotically stable
random walks.
\begin{corollary}
\label{C1}
Assume that $X\in\mathcal{D}(\alpha,\beta)$. If
\begin{equation}
\label{g-condition}
\int_1^\infty\frac{h(g(x))}{x h(c(x))}dx<\infty
\end{equation}
then \eqref{T1.1} is valid.
\end{corollary}
If $\mathbf{E}[(X^-)^2]<\infty$ then, clearly, $h(x)\sim cx$ and,
consequently, our condition \eqref{g-condition} coincides with
Novikov's integral test, see Theorem 1 in \cite{Nov82}. Note also
that Novikov \cite{Nov82} imposes a stronger assumption on $g$:
this function is assumed to be concave.

Recalling that $h$ is regularly varying of index $-\alpha(1-\rho)$,
we conclude that \eqref{T1.1} holds for every function $g$ bounded from
above by $x^\gamma$ with some $\gamma<\min\{1,1/\alpha\}$. But if $g(x)$ is
asymptotically equivalent to $c(x)$ then, applying the functional
limit theorem, we conclude that the statement of Corollary~\ref{C1}
is not valid. Therefore, condition \eqref{g-condition} is quite
close to the optimal one for asymptotically stable random walks.
\begin{remark}
Mogulskii and Pecherskii \cite{MP79} have derived the following
factorisation identity for $T_g$: If $g$ is superadditive, i.e.,
$g(x+y)\geq g(x)+g(y)$, then there
exists a sequence of events $E_n$ such that
\begin{equation}
\label{WH1}
\sum_{n=0}^\infty z^n\mathbf{P}(T_g>n)=
\exp\left\{\sum_{n=1}^\infty\frac{z^n}{n}\mathbf{P}(E_n)\right\}.
\end{equation}
Moreover, 
\begin{equation}
\label{WH2}
E_n\subseteq\{S_n\geq -g(n)\}\quad \text{ for all }n\geq1.
\end{equation}
Unfortunately, the structure of $E_n$ is highly non-trivial and
it is not clear how do determine the limit of $\mathbf{P}(E_n)$.
But using \eqref{WH2} one can obtain an asymptotically precise
upper bound for $\mathbf{P}(T_g>n)$. Indeed, it is immediate
from \eqref{WH2} that
$$
\mathbf{P}(T_g>n)\leq q_n,
$$
where $q_n$ is determined  by
$$
\sum_{n=0}^\infty z^nq_n=
\exp\left\{\sum_{n=1}^\infty\frac{z^n}{n}\mathbf{P}(S_n\geq -g(n))\right\}.
$$
Assume that function $g$ satisfies 
\begin{equation}
\label{new-test}
\int_1^\infty\frac{g(x)}{xc(x)}dx<\infty.
\end{equation}
Then applying the estimate
$$
\mathbf{P}(S_n\in[x,x+1))\leq\frac{C}{c(n)},
$$
we conclude that coefficients of
$$
R(z):=\exp\left\{\sum_{n=1}^\infty\frac{z^n}{n}\mathbf{P}(S_n\in[-g(n),0]\right\}
$$
are summable, i.e., $R(1)<\infty$. Noting now that 
$$
\sum_{n=0}^\infty z^nq_n=\left(\sum_{n=0}^\infty z^n\mathbf{P}(T_0>n)\right)R(z),
$$
we arrive at the relation
$$
q_n\sim R(1)\mathbf{P}(T_0>n)
$$
and, consequently,
\begin{equation}
\label{MP-bounds}
1\leq\liminf_{n\to\infty}\frac{\mathbf{P}(T_g>n)}{\mathbf{P}(T_0>n)}\leq
\liminf_{n\to\infty}\frac{\mathbf{P}(T_g>n)}{\mathbf{P}(T_0>n)}\leq R(1).
\end{equation}
Note also that in order to obtain the relation 
$\mathbf{P}(T_g>n)\sim C \mathbf{P}(T_0>n)$
from \eqref{WH1} it suffices to show that
$$
\sum_{n=1}^\infty\frac{1}{n}\left|\mathbf{P}(E_n)-\mathbf{P}(S_n>0)\right|<\infty.
$$

Recalling that $h$ is regularly varying of index $\alpha(1-\rho)$, we conclude
that \eqref{g-condition} is more restrictive than \eqref{new-test} for all
random walks with $\alpha(1-\rho)<1$. Thus, bounds \eqref{MP-bounds} imply that
$\mathbf{P}(T_g>n)$ can be of the same order as $\mathbf{P}(T_0>n)$ for
functions $g$ which do not satisfy \eqref{g-condition}.
\hfill$\diamond$
\end{remark}

The starting point of the proof of Theorem \ref{T1} is the following simple
observation: $T_g$ coincides with one of strict descending ladder epochs of $S_n$.
Let $(\tau_k,\chi_k)$ be independent copies of $(T_0,-S_{T_0})$. Then
\begin{equation}
\label{representation}
T_g=\sum_{k=1}^\nu \tau_k,
\end{equation}
where
\begin{equation}
\label{nu-def}
\nu:=\min\{k\ge 1: \chi_1+\cdots+\chi_k>g(\tau_1+\cdots+\tau_k)\}.
\end{equation}
Since the tail distribution function of $\tau$'s is regularly varying with index
$\rho-1\in(-1,0)$, we prove that, for any increasing function $g$,
\begin{equation}
\label{general}
\lim_{n\to\infty}\frac{\mathbf{P}(T_g>n)}{\mathbf{P}(T_0>n)}=\mathbf{E}\nu\in[1,\infty].
\end{equation}
Thus, to prove Theorem~\ref{T1} it suffices to show that $\mathbf{E}\nu$
is finite under the conditions of Theorem~\ref{T1}.

\vspace{6pt}

The subadditivity assumption in Theorem \ref{T1} seems to be purely
technical. Unfortunately, we do not know how to remove this
restriction. But for asymptotically stable random walks one can
replace \eqref{g-condition} by a stronger integral test which allows
to derive \eqref{T1.1} without subadditivity assumption.
\begin{theorem}
\label{T2}
Assume that $X\in\mathcal{D}(\alpha,\beta)$. If $g$ is increasing and
\begin{equation}
\label{g-condition2}
\int_1^\infty\frac{h(g(x))}{x h(c(x/\log x))}dx<\infty,
\end{equation}
then \eqref{T1.1} remains valid.
\end{theorem}

Integral test \eqref{g-condition2} is fulfilled for any increasing
regularly varying with index $\gamma<1/\alpha$ function $g$ and,
consequently, the subadditivity assumption is superflous for such
functions. On the other hand, it is not difficult to construct a regularly
varying with index $1/\alpha$ function $g$ which satisfies
\eqref{g-condition} but the integral in \eqref{g-condition2} becomes
infinite.

Our proof of Theorem \ref{T2} is based on an appropriate adaption of
the method from \cite{DW13}, where we studied the asymptotic behaviour
of exit times from cones of multidimensional random walks. The main
idea in \cite{DW13} was to use the classical universality for random
walks: If certain moments of a random walk are finite (this means that
random walk is sufficiently close to the Brownian motion) then the
asymptotics for the exit time of this random walk is the same, up to a
constant factor, as for the Brownian motion. In the proof of
Theorem~\ref{T3} we use a completely different type of universality:
We fix the distribution of a random walk and look for boundaries $g$
such that $T_g$ and $T_0$ have the same rate.

\vspace{6pt}

Now we turn to exit times from a shrinking domain. Define
$$
\widehat{T}_g:=\min\{n\geq1: S_n<g(n)\}.
$$
For $\widehat{T}_g$ one does not have any representation similar to
\eqref{representation}. For that reason there is no analogue of
Theorem~\ref{T1} for $\widehat{T}_g$. But one can look at the ratio
$\mathbf{P}(\widehat{T}_g>n)/\mathbf{P}(T_0>n)$  in the following way: If $g$
is positive then
$$
\frac{\mathbf{P}(\widehat{T}_g>n)}{\mathbf{P}(T_0>n)}
=\mathbf{P}(\widehat{T}_g>n|T_0>n)
$$
and one can try to represent the limit of this conditional probability as
a functional of $\{S_n\}$ conditioned to stay nonnegative. To formulate
the corresponding result we have to introduce some notation. It is
well-known that $h(x)$ is a positive harmonic function for $\{S_n\}$
killed at leaving $[0,\infty)$, that is,
$$
\mathbf{E}[h(x+X),x+X>0]=h(x),\quad x\geq0.
$$
We denote by $\mathbf{P}^h$ the Doob transform of $\mathbf{P}$ by the
function $h$. More precisely, $\mathbf{P}^h$ is the distribution of the
Markov chain on $[0,\infty)$ with transition function
$$
p^h(x,dy)=\frac{h(y)}{h(x)}\mathbf{P}(x+X\in dy),\quad x,y\geq0.
$$
The following statement is immediate from Lemma 2.5 in Afanasyev, Geiger,
Kersting and Vatutin \cite{AGKV}.
\begin{proposition}
\label{P6}
If \eqref{eq1} holds then
\begin{equation}
\label{T6.2}
\lim_{n\to\infty}\frac{\mathbf{P}(\widehat{T}_g>n)}{\mathbf{P}(T_0>n)}
=\mathbf{P}^h(\widehat{T}_g=\infty).
\end{equation}
\end{proposition}
Greenwood and Novikov \cite[Theorem 2]{GN} have shown the existence of the limit 
$\lim_{n\to\infty}\frac{\mathbf{P}(\widehat{T}_g>n)}{\mathbf{P}(T_0>n)}$
for random walks with $\mathbf{E}X=0$ and $\mathbf{E}e^{-\lambda X}<\infty$
for some $\lambda>0$. They have also shown that this limit is positive if and
only if $\mathbf{E} g(T_0)<\infty$. The information about the positivity of this
limit is very important, since if it is zero then the asymptotic behaviour of
$\mathbf{P}(\widehat{T}_g>n)$ remains unknown. Thus, in order to apply
Proposition~\ref{P6} we need to find reasonable conditions for the positivity of
$\mathbf{P}^h(\widehat{T}_g=\infty)$. Since
$\mathbf{P}^h(\widehat{T}_g=\infty)=\mathbf{P}^h(S_k\geq g(k)\text{ for all }k\geq1)$,
we infer that $\mathbf{P}^h(\widehat{T}_g=\infty)$ for every function $g$ such that
$\mathbf{P}^h(S_k<g(k)\text{ i.o.})=0$. One has also the reverse implication:
if $\mathbf{P}^h(S_k<g(k)\text{ i.o.})=1$ then $\mathbf{P}^h(\widehat{T}_g=\infty)=0$.

By Theorem 1 in Hambly, Kersting and Kyprianou \cite{HKK03} if
$\mathbf{E}X^2<\infty$ and $\mathbf{E}X=0$ then
\begin{align}
\label{hkk}
\mathbf{P}^h(S_k<g(k)\text{ i.o.})=0\text{ or }1\quad \text{ accordingly as }\quad
\int_1^\infty\frac{g(x)}{x^{3/2}}dx<\infty\text{ or }=\infty.
\end{align}
Consequently, for any oscillating random walk with finite variance,
$$
\mathbf{P}^h(\widehat{T}_g=\infty)>0\quad \Leftrightarrow\quad
\int_1^\infty\frac{g(x)}{x^{3/2}}dx<\infty.
$$
Unfortunately, we did not find results similar to \eqref{hkk} for random
walks with infinite variance. Thus, we do not have any criterion for the
positivity of $\mathbf{P}^h(\widehat{T}_g=\infty)$ for oscillating random
walks with infinite variance.

It turns out that the approach used in Theorem~\ref{T2} can be applied
to $\widehat{T}_g$ as well.
\begin{theorem}
\label{T3}
Assume that $X\in\mathcal{D}(\alpha,\beta)$ and that
\eqref{g-condition2} holds. Then there exists
$\widehat{V}(g)\in(0,\infty)$ such that, as $n\to\infty$,
\begin{equation}
\label{T3.1}
\mathbf{P}(\widehat{T}_g>n)\sim \widehat{V}(g)\mathbf{P}(T_0>n).
\end{equation}
\end{theorem}

\section{Proof of Theorem \ref{T1}}
We first derive \eqref{general}. If $\mathbf{E}\nu$ is finite then
the desired relation follows immediately from Theorem 2 in Greenwood
and Monroe \cite{GM77}. Assume that $\mathbf{E}\nu=\infty$. Since
$\tau$ is positive,
$$
\mathbf{P}(T_g>n)\geq\mathbf{P}(\sum_{k=1}^{\nu\wedge m}\tau_k>n)
$$
for any $m\geq1$. Applying Theorem 2 from \cite{GM77} to the stopping
time $\nu\wedge m$, we obtain
$$
\liminf_{n\to\infty}\frac{\mathbf{P}(T_g>n)}{\mathbf{P}(T_0>n)}
\geq\mathbf{E}[\nu\wedge m].
$$
Letting $m$ go to infinity, we complete the derivation of \eqref{general}.

As it was already mentioned in the introduction, Theorem \ref{T1} 
follows from \eqref{general} and the following statement.
\begin{proposition}
\label{finite_mean}
Under the assumptions of Theorem \ref{T1},
\begin{equation}\label{P1.1}
\mathbf{E}\nu<\infty.
\end{equation}
\end{proposition}

\subsection{Proof of Proposition \ref{finite_mean} for random walks with $\mathbf{E}\chi=\infty$.}
It follows from the subadditivity of $h(g(x))$ that
$$
\nu\leq\mu:=\min\{k\ge 1: h(\chi_1+\ldots+\chi_k)>h(g(\tau_1))+\ldots+h(g(\tau_k))\}.
$$
Thus, it suffices to show that the expectation of $\mu$ is finite.
But the latter is a consequence of the positiv recurrence of zero
for the sequence
$$
Z_n:=\left(Z_0+\sum_{k=1}^n h(g(\tau_k))-h\left(\sum_{k=1}^n\chi_k\right)\right)^+,\quad n\geq1.
$$

According to Lemma 1 and Proposition in Erickson \cite{Erick73},
our assumption $\mathbf{E}\chi=\infty$ is equivalent to
$\mathbf{E}h(\chi)=\infty$. Consequently,
$$
\mathbf{E}[h(y+\chi)-h(y)]=\infty
$$
for every $y\geq0$. Combining this with \eqref{hg-moment}, we finally get
\begin{equation}
\label{moment}
\mathbf{E}[h(y+\chi)-h(y)-h(g(\tau))]=\infty
\end{equation}
for every $y\geq0$.

Fix some positive $x$ and $y$ such that $x>h(y)$. Then
\begin{align*}
&\mathbf{E}\left[Z_{n+1}-Z_n\Big|\sum_{k=1}^n h(g(\tau_k))=x,\sum_{k=1}^n\chi_k=y\right]\\
&\hspace{1cm}\leq \mathbf{E}[h(y)-h(y+\chi)+h(g(\tau));x+h(g(\tau))-h(y+\chi)>0]\\
&\hspace{1cm}=-\mathbf{E}[h(y+\chi)-h(y)-h(g(\tau));h(y+\chi)-h(y)-h(g(\tau))<x-h(y)].
\end{align*}
Taking into account \eqref{moment}, we conclude that there exists $A>0$ such that 
$$
\mathbf{E}\left[Z_{n+1}-Z_n\big|\sum_{k=1}^n h(g(\tau_k))=x,\sum_{k=1}^n\chi_k=y\right]\leq-1
$$
for all $x,y$ satisfying $x-h(y)>A$. This implies that the hitting time of $[0,A]$ by the
sequence $Z_n$ has finite mean. Since $0$ can be approached from any point of $[0,A]$ in a
finite time, we infer that $\mathbf{E}\mu<\infty$.
\subsection{Proof of Proposition \ref{finite_mean} for random walks with $\mathbf{E}\chi<\infty$.}
If random walk $S_n$ is such that $\mathbf{E}\chi<\infty$ then $h(x)\sim Cx$ as $x\to\infty$ and, in particular,
$\mathbf{E}h(\chi)<\infty$. Therefore, we can not use the sequence $Z_n$ from the previous subsection.
Fix some $k\geq1$ and consider 
$$
Z^{(k)}:=\chi_1+\chi_2+\ldots+\chi_k-g(\tau_1+\tau_2+...+\tau_k)
$$
and let $Z_i^{(k)}$ be independent copies of $Z$.  Using the subadditivity of $g$ once again, one can
easily show that
$$
\nu\leq k\mu^{(k)},
$$
where
$$
\mu^{(k)}:=\min\{n\geq1:Z_1^{(k)}+Z_2^{(k)}+\ldots+Z_n^{(k)}>0\}.
$$
From this we conclude that Proposition \ref{finite_mean} will follow from the existence of $k$
such that $\mathbf{E}Z^{(k)}>0$. Therefore, it suffices to show that
\begin{equation}
\label{suff.cond}
\lim_{k\to\infty}\frac{1}{k}\mathbf{E}g(\tau_1+\tau_2+...+\tau_k)=0.
\end{equation}

Set $U_k=\tau_1+\tau_2+...+\tau_k$ and note that $\tau_i\in\mathcal{D}(1-\rho,1)$.
Let $a(k)$ be the corresponding norming sequence, that is, $U_k/a(k)$ converges
weakly towards a stable distribution of index $1-\rho$.

Since $g$ is increasing and sublinear,
\begin{align*}
\mathbf{E}g(U_k)&=\sum_{j=0}^\infty\mathbf{E}[g(U_k);U_k\in[ja(k),(j+1)a(k))]\\
&\leq g(a(k))+2\sum_{j=1}^\infty g(ja(k))\mathbf{P}(U_k>ja(k)).
\end{align*}
Since the tail of $\tau$ is regularly varying with index $\rho-1>-1$,
$$
\mathbf{P}(U_k>x)\leq Ck\mathbf{P}(\tau>x),\quad x\geq0.
$$
Consequently,
\begin{align}
\label{1.step}
\nonumber
\mathbf{E}g(U_k)&\leq g(a(k))+C\sum_{j=1}^\infty g(ja(k))k\mathbf{P}(\tau>ja(k))\\
&\leq g(a(k))+Ck\mathbf{E}[g(\tau);\tau>a(k)].
\end{align}
It follows from the assumption $\mathbf{E} g(\tau)<\infty$ that
$$
\frac{g(a_k)}{k}\leq C g(a_k)\mathbf{P}(\tau>a(k))
\leq C\mathbf{E}[g(\tau);\tau>a(k)].
$$
Furthermore,
$$
\mathbf{E}[g(\tau);\tau>a(k)]\to0\quad\text{as }k\to\infty.
$$
Applying these relations to \eqref{1.step}, we get \eqref{suff.cond}.
This completes the proof of Proposition \ref{finite_mean}.
\subsection{Proof of Corollary \ref{C1}} It suffices to show that 
\eqref{hg-moment} and \eqref{g-condition} are equivalent. First we
note that if $X\in\mathcal{D}(\alpha,\beta)$ then
\begin{equation*}
\mathbf{P}(\tau_1>x)\sim C\frac{1}{h(c(x))}\quad\text{as }x\to\infty
\end{equation*}
and, consequently, \eqref{g-condition} is equivalent to 
$$
\int_1^\infty\frac{h(g(x))}{x}\mathbf{P}(\tau_1>x)dx<\infty.
$$
Consequently,
\begin{equation*}
\mbox{ \eqref{g-condition} is equivalent to}\quad
\mathbf{E}h(g(\tau_1))<\infty.
\end{equation*}

\section{Proof of Theorem \ref{T2}}
\subsection{Finiteness of $V(g)$}
Let $h$ be the renewal function of the decreasing ladder height
process. It is known that it is harmonic for $S_n$ killed at leaving
$\mathbb{R}^+$, that is, 
$$
\e[h(x+X_1);x+X_1>0] = h(x),\quad x>0.
$$
Extending $h(x)$ to the negative half line with $0$ we can
write this equality as $\e h(x+X_1)=h(x)$ for $x>0.$ 
\begin{lemma}\label{lem1}
The sequence $Y_n = h(S_n+g(n)) 1\{T_g>n\}$ is a submartingale. 
\end{lemma}
\begin{proof}
Clearly,
\begin{align*}
\e\left[Y_{n+1}-Y_n|\mathcal{F}_n\right]
&=\e\left[\left(h(S_{n+1}+g(n+1))-h(S_{n}+g(n))\right){\rm 1}\{T_g>n\}|\mathcal{F}_n\right]\\
&\hspace{2cm}-\e\left[h(S_{n+1}+g(n+1)){\rm 1}\{T_g=n+1\}|\mathcal{F}_n\right].
\end{align*}
Now note that $h(S_{n+1}+g(n+1)){\rm 1}\{T_g=n+1\} =0$ since 
$h(x)=0$ for $x\le 0$. Next, by harmonicity of $h$ 
\begin{align*}
&\e\left[Y_{n+1}-Y_n|\mathcal{F}_n\right]\\
&\hspace{1cm}=\e\left[\left(h(S_{n+1}+g(n+1))\right){\rm 1}\{T_g>n\}|\mathcal{F}_n\right]
-h(S_{n}+g(n)) {\rm 1}\{T_g>n\}\\
&\hspace{1cm}= h(S_{n}+g(n+1)) {\rm 1}\{T_g>n\} - 
h(S_{n}+g(n)) {\rm 1}\{T_g>n\} \ge 0,
\end{align*}
since $h$ and $g$ are monotone increasing.
\end{proof}

Fix any positive sequence $\varepsilon_n\to 0$ and denote 
$$
\nu_{n}:=\min\{k\ge 1:|S_{k}|\ge c(\varepsilon_n n) \}.
$$
\begin{lemma}\label{lem.concentration}
There exists a constant $C$ such that
$$
\pr(\nu_n >\delta n) \le e^{-C\delta/\varepsilon_n}
$$
for all $\delta>0$.
\end{lemma}

\begin{proposition}\label{V_exists}
Assume that $g$ is such that  $g(0)>0$ and 
condition \eqref{g-condition} holds. Then there exists a finite strictly positive limit 
$$
V(g) = \lim_n \e \left[h(S(n)+g(n));T_g>n\right]<\infty.
$$
\end{proposition}

\begin{proof}

Since $h(S(n)+g(n)) {\rm 1}\{T_g>n\} $ is a submartingale it is
sufficient to show that 
$$
\sup_n \e \left[h(S(n)+g(n));T_g>n\right]<\infty.
$$
By the well-known submartingale convergence theorem this will imply 
the statement of the proposition. 

Fix some $n_0>1$ and define
$$
n_m=n_0[(1/\delta)^m],\quad m\geq1,
$$
where $[r]$ denotes the integer part of $r$. Since $h(S_n+g(n)){\rm 1}\{T_g>n\}$
is a positive submartingale, it suffices to show that the subsequence
$\mathbf{E}[h(S_{n_m}+g(n_m));T_g>n_m]$ is bounded.
We first split the  expectation into 2 parts,
\begin{align*}
&\e [h(S_{n_m}+g(n_m));T_g>n_m]=E_{1}+E_{2}\\
&\hspace{1cm}:=\e \left[h(S_{n_m}+g(n_m));T_g>n_m,\nu_{n_m}\leq n_{m-1}\right]\\
&\hspace{2cm}+\e\left[h(S_{n_m}+g(n_m));T_g>n_m,  \nu_{n_m}> n_{m-1}\right].
\end{align*}
For fixed $n$ let 
$$ 
\widetilde \tau (g(n)):=\min \{k\ge 1: S_k+g(n) \le 0\}.
$$
For the second expectation note 
\begin{align*}
  E_2 =  \e\left[ \e\left[h(S_{n_m}+g(n_m))  {\rm 1}\{T_g>n_m\} \mid \mathcal
    F_{n_{m-1}}\right]{\rm 1}\{\nu_{n_m}> n_{m-1}\}\right].
\end{align*}
Then, using the harmonicity of $h$, 
\begin{align*}
&\e\left[h(S_{n_m}+g(n_m)){\rm 1}\{T_g>n_m\} \mid \mathcal F_{n_{m-1}}\right]\\
&\hspace{1cm}\le \e_{S_{n_{m-1}}} [h(S_{n_m-n_{m-1}}+g(n_m)){\rm 1}\{\widetilde\tau(g(n_m))>n_m-n_{m-1}\}]{\rm 1}\{T_g>n_{m-1}\}\\
&\hspace{1cm}\le h(S_{n_{m-1}}+g(n_m))\\
&\hspace{1cm}\le h(S_{n_{m-1}}+g(n_{m-1})) +h(g(n_m)),
\end{align*}
where we used the subadditivity of the renewal function in the last
step. Therefore,
\begin{align*}
E_2 &\le \e [h(S_{n_{m-1}})+g(n_{m-1})) +h(g(n_m));\nu_{n_m}>n_{m-1} ]\\
&\le  \e\left[h(c(\varepsilon_{n_m} n_m)+g(n_m))+h(g(n_m));\nu_{n_m}> n_{m-1}\right]\\
&\le \left(h(c(\varepsilon_{n_m} n_m))+2h(g(n_m))\right)\pr(\nu_{n_m}> n_{m-1}).
\end{align*}
Applying Lemma \ref{lem1} with $\varepsilon_n:=(3\delta C\log n)^{-1}$, we obtain
\begin{align*}
E_2\leq \left(h(c(\varepsilon_{n_m} n_m))+2h(g(n_m))\right)n_m^{-3}.
\end{align*}

The first expectation can be estimated similarly. First note 
\begin{align*}
  E_1 =  \e\left[ \e\left[h(S_{n_m}+g(n_m))  {\rm 1}\{T_g>n_m\} \mid \mathcal
    F_{\nu_{n_m}}\right], \nu_{n_m}\le n_{m-1}\right].
\end{align*}
Then,  again, using the harmonicity of $h$, 
\begin{align*}
&\e\left[h(S_{n_m}+g(n_m)){\rm 1}\{T_g>n_m\} \mid \mathcal F_{\nu_{n_m}}\right]\\
&\hspace{1cm}\le \e_{S_{\nu_{n_m}}} [h(S_{n_m-\nu_{n_m}}+g(n_m))
{\rm 1}\{\widetilde\tau(g(n_m))>n_m-\nu_{n_m}\}]{\rm 1}\{T_g>\nu_{n_m}\}\\
&\hspace{1cm}\le h(S_{\nu_{n_m}}+g(n_m)){\rm 1}\{T_g>\nu_{n_m}\}\\
&\hspace{1cm}\le \left(h(S_{\nu_{n_m}}+g(\nu_{n_m})) +h(g(n_m))\right){\rm 1}\{T_g>\nu_{n_m}\}.
\end{align*}
Now since $S_{\nu_n}\ge c(\varepsilon_n n)$ we have, 
\begin{align*}
  &\e\left[h(S_{n_m}+g(n_m))  {\rm 1}\{T_g>n_m\} \mid \mathcal
    F_{\nu_{n_m}}\right]\\
 &\hspace{1cm}\le h(S_{\nu_{n_m}}+g(\nu_{n_m})) \left(
    1+\frac{h(g(n_m))}{h(c(\varepsilon_{n_m} n_m))}\right){\rm 1}\{T_g>\nu_{n_m}\}.
\end{align*}
Hence, using the latter inequality and using 
the submartingale property we have, 
\begin{align*}
  E_1 &\le \left(
    1+\frac{h(g(n_m))}{h(c(\varepsilon_{n_m} n_m))}  \right) 
  \e [h(S_{\nu_{n_m}}+g(\nu_{n_m}));T_g>\nu_{n_m},\nu_{n_m}\le n_{m-1}]\\
&\le 
\left(
    1+\frac{h(g(n_m))}{h(c(\varepsilon_{n_m} n_m))}  \right) 
 \e [h(S_{n_{m-1}}+g(n_{m-1}));T_g> n_{m-1}].
\end{align*}
As a result we have 
\begin{multline*}
\e [h(S_{n_m}+g(n_m));T_g>n_m]\\
\le \left(1+\frac{h(g(n_m))}{h(c(\varepsilon_{n_m} n_m))}\right) 
 \e [h(S_{n_{m-1}}+g(n_{m-1}));T_g> n_{m-1}]\\
+\left(h(c(\varepsilon_{n_{m}} n_{m}))+2h(g(n_m))\right) / n_m^3.
\end{multline*}
Iterating this procedure $m$ times we obtain 
\begin{align*}
  &\e [h(S_{n_m}+g(n_m));T_g>n_m]  \le 
\prod_{j=1}^m \left(
    1+\frac{h(g(n_j))}{h(c(\varepsilon_{n_j} n_j))}  \right) \\
&\hspace{0.5cm}\times
\left( \e [h(S_{n_0}+g(n_0));T_g>n_0]  + \sum_{j=1}^m\left(h(c(\varepsilon_{n_j} n_j))+2h(g(n_{j}))\right) / n_j^3\right).
\end{align*}
Now note \eqref{g-condition2} implies that
$$
\sum_{j=1}^\infty \frac{h(g(n_j))}{h(c(\varepsilon_{n_j} n_j))} <\infty. 
$$ 
Therefore, the product
$$
\prod_{j=1}^\infty \left(1+\frac{h(g(n_j))}{h(c(\varepsilon_{n_j} n_j))}  \right) 
$$
is finite. Furthermore, recalling that the index of $h$ is $\alpha(1-\rho)$, we conclude that
$$
\sum_{j=0}^\infty \frac{(h(c(\varepsilon_{n_j} n_j))+2h(g(n_{j}))}{n_j^3}<\infty.
$$

Then the statement of the proposition immediately follows.

\end{proof}

\subsection{Derivation of the asymptotics}

In the proof of \eqref{T1.1} we will require the following result.
\begin{lemma}\label{lem2}
Let $S_n$ be asymptotically stable random walk. Then, 
\begin{equation}\label{lem2.eq1}
\pr(T_x>n)\sim h(x)n^{\rho-1}L(n)
\end{equation}
uniformly in $x$ such that $x/c(n)\to 0$. 
In  addition the following estimate is valid 
for all $x\ge 0$,
\begin{equation}\label{lem2.eq2}
\pr(T_x>n)\le  Ch(x)n^{\rho-1}L(n).
\end{equation}
\end{lemma}

\begin{proof}

The first statement \eqref{lem2.eq1} is Corollary~3 of \cite{D12}.  

Denote 
$$
\sigma(x):=\min\left\{k\ge 1: \sum_{i=1}^k\chi_i>x\right\}.
$$
Then,
\begin{align*}
\pr(T_x>n)&\le \pr\left(\sum_{k=1}^{\sigma(x)} \tau_k>n\right)
\le \pr\left(\sum_{k=1}^{\sigma(x)} \tau_k\wedge n >n\right)\\
&\le \frac{\e \left[ \sum_{k=1}^{\sigma(x)} \tau_k\wedge n \right]}{n} 
=\frac{\e \sigma(x) \e[T_0\wedge n]}{n}, 
\end{align*}
where we applied the Wald's identity in the last step.  Since the
tail of $T_0$ is regularly varying of order $\rho-1 \in (0,1)$,
by the Tauberian theorem, $\e[T_0\wedge n] \sim \rho^{-1} n \pr(\tau>n)$.
Also note that  $h(x)=\e \sigma(x)$. Hence, as $n\to\infty$, uniformly in $x$, 
$$
\pr(T_x>n) \le (1+o(1))h(x)\rho^{-1}\pr(T_0>n).
$$
\end{proof}

According to \eqref{g-condition2},
$$
\int_x^{2x}\frac{h(g(y))}{yh(c(y/\log y))}dy\to0\quad\text{as }x\to\infty.
$$
Since functions $c(x)$ $h(x)$, $g(x)$ are increasing and $h$, $c$ are
in addition regularly varying, we conclude that
$$
\frac{h(g(x))}{h(c(x/\log x)}\to0.
$$
In other words,
$$
g(x)=o(c(x/\log x)).
$$
Consequently, there exists a sequence $\delta_n\to0$ such that
$$
g(n)=o(c(\delta_n n/\log n)).
$$
Moreover, we may asume that $\delta_n$ is such that slowly  varying
functions in $\pr(\tau>n)\sim n^{\rho-1}L(n)$ and 
$h(x)=x^{\alpha(1-\rho)}l(x)$ are such that $l(\delta_n n)\sim l(n)$ 
and $L(\delta_n n )\sim L(n)$.

Applying Lemma \ref{lem.concentration} with $\delta=\delta_n$ and
$\varepsilon_n=\delta_n/(2C\log n)$, we get
\begin{equation}
\label{1.split}
\pr(T_g>n) = \pr(T_g>n,\nu_n <n\delta_n) +o(n^{-2}). 
\end{equation}

This immediately gives us the bound from above. Indeed, 
\begin{align*}
 &\pr(T_g>n,\nu_n\le n\delta_n ) \\
&\le \int_{\varepsilon c_{n\delta_n}}^\infty \pr (T_g>\nu_n, S_{\nu_n}
\in dy,\nu_n\le n\delta_n )\pr(\widetilde \tau_{y+g(n)}>n-n\delta_n)\\
&\le 
C\int_{\varepsilon c_{n\delta_n}}^\infty \pr (T_g>\nu_n, S_{\nu_n}
\in dy,\nu_n\le n\delta_n )h(y+g(n)) 
(n-n\delta_n)^{\rho-1}L(n-n\delta_n),
\end{align*}
where we applied \eqref{lem2.eq2} in the last step.
Now note that, uniformly in $y>c(\delta_n n/\log n)$, 
$h(y+g(n)) \sim h(y)$ and
$(n-n\delta_n)^{\rho-1}L(n-n\delta_n)\sim n^{\rho-1}L(n)$.
Therefore, for large $n$ we have, 
\begin{align*}
 \pr(T_g>n,\nu_n <n \delta_n)  &\le 
C \int_{\varepsilon c_{n\delta_n}}^\infty \pr (T_g>\nu_n, S_{\nu_n}
\in dy,\nu_n\le n\delta_n ) h(y) \pr(T_0>n)\\
&\le C \e [h(S_{\nu_n});T_g>\nu_n,\nu_n\leq n\delta_n] \pr(T_0>n).
\end{align*}
By the submartingale property and Proposition~\ref{V_exists},
\begin{equation}\label{new_UB}
\e [h(S_{\nu_n});T_g>\nu_n,\nu_n\leq n\delta_n]
\leq \e [h(S_{n\delta_n});T_g>n\delta_n]\leq V(g).
\end{equation}
Consequently,
\begin{align*}
\pr(T_g>n,\nu_n <n \delta_n)\le C V(g) \pr(T_0>n).
\end{align*}
Combining this with \eqref{1.split}, we get
\begin{equation}
\label{UpperBound}
\mathbf{P}(T_g>n)\leq CV(g)\mathbf{P}(T_0>n).
\end{equation}

Next we are going to
improve the latter bound to obtain the sharp asymptotics. 
Let $A_n\uparrow \infty$ be such that $A_nc(\delta_n n/\log n)=o(c(n)).$ 
Set, for brevity, $r_n:=c(\delta_n n/\log n)$. We split the probability
$\pr(T_g>n,\nu_n\leq \delta_n n)$ into two parts:
\begin{align}
\label{split}
\nonumber
  \pr(T_g>n,\nu_n\leq \delta_n n) 
= &\pr\left(T_g>n,\nu_n <\delta n, S_{\nu_n}\in (r_n, A_n r_n)\right) \\
&+\pr(T_g>n,\nu_n <\delta n, S_{\nu_n}> A_n r_n).
\end{align}
For the first term we are gong to apply \eqref{lem2.eq1}.  We have, 
\begin{align}
\label{up.bound}
\nonumber
&\pr\left(T_g>n,\nu_n <\delta n, S_{\nu_n}\in (r_n, A_nr_n)\right)\\
\nonumber
&\hspace{1cm}\le \int_{r_n}^{A_n r_n} \pr (T_g>\nu_n, S_{\nu_n}
\in dy,\nu_n\le n\delta_n )\pr(\widetilde \tau_{y+g(n)}>n-n\delta_n)\\
\nonumber
&\hspace{1cm}\sim\int_{r_n}^{A_n r_n}\pr(T_g>\nu_n, S_{\nu_n}\in dy,\nu_n\le n\delta_n )
h(y)\pr\{T_0>n\}\\
\nonumber
&\hspace{1cm}\le \e [h(S_{\nu_n});T_g>\nu_n,\nu_n\leq n\delta_n] \pr\{T_0>n\}\\
&\hspace{1cm}\leq V(g)\pr\{T_0>n\},
\end{align}
where we used \eqref{new_UB} in the last step. The bound
from below can be obtained by similar arguments:
\begin{align}
\label{l.bound}
\nonumber
&\pr\left(T_g>n,\nu_n <\delta n, S_{\nu_n}\in (r_n, A_nr_n)\right)\\
\nonumber
&\ge \int_{r_n}^{A_n r_n} \pr (T_g>\nu_n, S_{\nu_n}
\in dy,\nu_n\le n\delta_n )\pr(\widetilde \tau_{y}>n-n\delta_n)\\
\nonumber
&\sim\int_{r_n}^{A_n r_n}\pr(T_g>\nu_n, S_{\nu_n}\in dy,\nu_n\le n\delta_n )
h(y)\pr(T_0>n)\\
&= \e [h(S_{\nu_n});T_g>\nu_n, S_{\nu_n}<A_nr_n, \nu_n\le n\delta_n] \pr(T_0>n).
\end{align}

Now we turn to the second term in \eqref{split}.
Using the Markov property and \eqref{lem2.eq1}, we obtain
\begin{multline*}
  \pr(T_g>n,\nu_n <\delta_n n, S_{\nu_n}> A_n r_n,T_g>\nu_n) \\
\le C \e [h(S_{\nu_n});S_{\nu_n}> A_n r_n,\nu_n <\delta_n n,T_g>\nu_n] \pr(T_0>n).
\end{multline*}
So we are left to prove that 
\begin{equation}
\label{tail.expectation}
\e [h(S_{\nu_n});S_{\nu_n}> A_n r_n,\nu_n <\delta_n n,T_g>\nu_n] \to 0. 
\end{equation}
Now  note that $S_{\nu_n-1}\le r_n $. Then on
the event $\{S_{\nu_n}> A_n r_n\}$ we have 
$$
X_{\nu_n}>(A_n-1) r_n >A_n r_n/2>S_{\nu_n-1}.
$$ 
Hence 
$$S_{\nu_n} = X_{\nu_n} + S_{\nu_n-1} \le 2 X_{\nu_n}.$$
Using the subadditivity and monotone increase of $h$ we obtain 
$$
h(S_{\nu_n})\le h(2 X_{\nu_n})\le 2 h(X_{\nu_n}).
$$
This implies that
\begin{align*}
 &E:=\e [h(S_{\nu_n});S_{\nu_n}> A_n r_n,\nu_n <\delta_n n,T_g>\nu_n] \\
&\hspace{2cm}\le  2 \e [h(X_{\nu_n});X_{\nu_n}> (A_n /2)r_n,\nu_n <\delta_n n,T_g>\nu_n]\\
&\hspace{2cm}\le\sum_{k=1}^{\delta_n n} \e [h(X_{k});X_{k}> (A_n /2)r_n,\nu_n =k,T_g>k]\\ 
&\hspace{2cm}\le 2 \sum_{k=1}^{\delta_n n}  \pr(T_g\geq k)\e [h(X);X> (A_n /2)c_{n\delta_n}]
\end{align*}
Applying \eqref{UpperBound}, we obtain
\begin{align*}
E &\le C \sum_{k=1}^{\delta_n n}\pr(T_0\geq k)\e [h(X);X> (A_n/2)r_n]\\
&\le C \delta_n n \pr(T_0>\delta_n n)\e [h(X);X> (A_n/2)r_n].
\end{align*}
Recall that if $X\in\mathcal{D}(\alpha,\beta)$ then the distribution function
of $\theta(du):=u^2\pr(|X|\in du)$ is regularly varying with index $2-\alpha$,
that is, 
$$
\Theta(x):=\theta((0,x))=x^{2-\alpha}\ell(x).
$$
Therefore,
\begin{align*}
\e [h(X);X> (A_n/2)r_n]\leq \int_{(A_n/2)r_n}^\infty\frac{h(x)}{x^2}\theta(dx)
\leq C\frac{h(A_n r_n)}{(A_nr_n)^2}\Theta(A_nr_n).
\end{align*}
We can choose $A_n$ in such a way that
$$
\frac{h(A_n r_n)}{(A_nr_n)^2}\Theta(A_nr_n)
\sim\left(\frac{c(n)}{A_nr_n}\right)^{\alpha\rho} \frac{h(c(n))}{c^2(n)}\Theta(c(n)).
$$
Consequently,
\begin{align*}
  E\le C \delta_n^\rho n \pr(T_0>n) \left(\frac{c(n)}{A_nr_n}\right)^{\alpha\rho} \frac{h(c(n))}{c^2(n)}\Theta(c(n)).
\end{align*}
Now recall that 
$$
h(c_n)\pr(T_0>n)\to C_0,\quad n \frac{\Theta(c(n)}{c^2(n)} \to 1
$$
for some constant $C_0$. Hence, for an appropriate choice of $A_n$,
$$
E\le C \delta_n^\rho \left(\frac{c(n)}{A_nr_n}\right)^{\alpha\rho}\to0
$$
as $n\to \infty$. This completes the proof of \eqref{tail.expectation}.

Noting that \eqref{tail.expectation} yields 
$$
\pr(T_g>n,\nu_n <\delta_n n, S_{\nu_n}> A_n r_n,T_g>\nu_n)
=o(\pr(T_0>n))
$$
and taking into account \eqref{up.bound}, we get
\begin{equation}
\label{up.bound2}
 \limsup_{n\to\infty}\frac{\mathbf{P}(T_g>n)}{\mathbf{P}(T_0>n)}\leq V(g).
\end{equation}

Combining \eqref{l.bound} and \eqref{tail.expectation}, we have
\begin{equation}
\label{l.bound2}
 \liminf_{n\to\infty}\frac{\mathbf{P}(T_g>n)}{\mathbf{P}(T_0>n)}\geq 
\liminf_{n\to\infty}\e [h(S_{\nu_n});T_g>\nu_n,\nu_n\le n\delta_n].
\end{equation}

It follows from Lemma \ref{lem.concentration} with $\delta=\delta_n$ and
$\varepsilon_n=\delta_n/(2C\log n)$ that
$$
\e[h(S_{\delta_n n}+g(\delta_n n));T_g>\delta_n n,\nu_n> n\delta_n]\leq 
g(c(n))\mathbf{P}(\nu_n> n\delta_n)\to 0.
$$
Consequently,
$$
\e [h(S_{\nu_n});T_g>\nu_n,\nu_n\le n\delta_n]=
\e [h(S_{\theta_n}+g(\theta_n));T_g>\theta_n]+o(1),
$$
where $\theta_n:=\nu_n\wedge \delta_nn$.
Applying the optional stopping theorem to the submartingale
$h(S_n+g(n)){\rm 1}\{T_g>n\}$, we get, for every fixed $N$,
$$
\e [h(S_{\theta_n}+g(\theta_n));T_g>\theta_n]\geq 
\e [h(S_{\theta_n\wedge N}+g(\theta_n\wedge N));T_g>\theta_n\wedge N].
$$
Note also that
\begin{align*}
&\Big|\e[h(S_{\theta_n\wedge N}+g(\theta_n\wedge N));T_g>\theta_n\wedge N,\nu_n<N]\\
&\hspace{3cm}-\e[h(S_N+g(N));T_g>N,\nu_n<N]\Big|\\
&\hspace{2cm}\leq \e[h(\max_{k<N}|S_k|+g(N));\max_{k<N}|S_k|>c(\delta_n n/\log n)]\to0.
\end{align*}
As a result,
$$
\e [h(S_{\theta_n}+g(\theta_n));T_g>\theta_n]\geq \e [h(S_{N}+g(N));T_g>N]+o(1)
$$
and, consequently,
$$
\liminf_{n\to\infty}\e [h(S_{\nu_n});T_g>\nu_n,\nu_n\le n\delta_n]
\geq \e [h(S_{N}+g(N));T_g>N].
$$
Letting $N\to\infty$ we get
$$
\liminf_{n\to\infty}\e [h(S_{\nu_n});T_g>\nu_n,\nu_n\le n\delta_n]
\geq V(g).
$$
Combining this with \eqref{up.bound2}, we complete the proof of
Theorem \ref{T2}.

\section{Proof of Theorem \ref{T3}}
The proof follows closely the proof of Theorem~\ref{T2}. 
\subsection{Positivity of $\widehat V(g)$}
\begin{lemma}\label{hat.lem1}
The sequence $Y_n = h(S_n-g(n)) 1\{\widehat T_g>n\}$ is a supermartingale. 
\end{lemma}
\begin{proof}
Clearly,
\begin{align*}
\e\left[Y_{n+1}-Y_n|\mathcal{F}_n\right]
&=\e\left[\left(h(S_{n+1}-g(n+1))-h(S_{n}-g(n))\right){\rm 1}\{\widehat T_g>n\}|\mathcal{F}_n\right]\\
&\hspace{2cm}-\e\left[h(S_{n+1}-g(n+1)){\rm 1}\{\widehat  T_g=n+1\}|\mathcal{F}_n\right].
\end{align*}
Now note that $h(S_{n+1}-g(n+1)){\rm 1}\{\widehat  T_g=n+1\} =0$ since 
$h(x)=0$ for $x\le 0$. Next, by harmonicity of $h$ 
\begin{align*}
&\e\left[Y_{n+1}-Y_n|\mathcal{F}_n\right]\\
&\hspace{1cm}=\e\left[\left(h(S_{n+1}-g(n+1))\right){\rm 1}\{\widehat
  T_g>n\}|\mathcal{F}_n\right]
-h(S_{n}-g(n)) {\rm 1}\{\widehat  T_g>n\}\\
&\hspace{1cm}= h(S_{n}-g(n+1)) {\rm 1}\{\widehat  T_g>n\} - 
h(S_{n}-g(n)) {\rm 1}\{\widehat  T_g>n\} \le 0,
\end{align*}
since $h$ and $g$ are monotone increasing.
\end{proof}
By this supermartingale property, the limit
$\lim_{n\to\infty}\mathbf{E}[h(S_n-g(n);\widehat{T}_g>n]$ exists and is finite.
Thus, we need to show that this limit is positive.
\begin{proposition}
Assume that $g$ is such that  $g(0)>0$ and 
condition \eqref{g-condition} holds. Then there exists a
strictly positive and finite  limit 
$$
\widehat  V(g):= \lim_{n\to\infty} \e \left[h(S(n)-g(n));\widehat  T_g>n\right].
$$
\end{proposition}

\begin{proof}

Since $h(S(n)-g(n)) {\rm 1}\{\widehat  T_g>n\} $ is a supermartingale it is
sufficient to show that 
$$
\inf_n \e \left[h(S(n)-g(n));\widehat T_g>n\right]>0.
$$

Fix some $n_0>1$ and define
$$
n_m=n_0[(1/\delta)^m],\quad m\geq1,
$$
where $[r]$ denotes the integer part of $r$. Since 
$\mathbf{E}[h(S_n-g(n));\widehat T_g>n]$ is decreasing,
it suffices to show that
$\inf_m \mathbf{E}[h(S_{n_m}-g(n_m));\widehat T_g>n_m]>0$.
First 
\begin{align*}
&\e [h(S_{n_m}-g(n_m));\widehat T_g>n_m]\ge E_{1}\\
&\hspace{1cm}:=\e \left[h(S_{n_m}-g(n_m));\widehat T_g>n_m,\nu_{n_m}\leq n_{m-1}\right].
\end{align*}
For fixed $n$ let 
$$ 
\widetilde \tau (g(n)):=\min \{k\ge 1: S_k-g(n) \le 0\}.
$$
Then 
\begin{align*}
  E_1 =  \e\left[ \e\left[h(S_{n_m}-g(n_m))  {\rm 1}\{\widehat T_g>n_m\} \mid \mathcal
    F_{\nu_{n_m}}\right], \nu_{n_m}\le n_{m-1}\right].
\end{align*}
Then,  again, using the harmonicity of $h$, 
\begin{align*}
&\e\left[h(S_{n_m}-g(n_m)){\rm 1}\{\widehat  T_g>n_m\} \mid \mathcal F_{\nu_{n_m}}\right]\\
&\hspace{1cm}\ge \e_{S_{\nu_{n_m}}} [h(S_{n_m-\nu_{n_m}}-g(n_m))
{\rm 1}\{\widetilde\tau(g(n_m))>n_m-\nu_{n_m}\}]{\rm 1}\{\widehat T_g>\nu_{n_m}\}\\
&\hspace{1cm}\ge h(S_{\nu_{n_m}}-g(n_m)){\rm 1}\{\widehat T_g>\nu_{n_m}\}\\
&\hspace{1cm}\ge \left(h(S_{\nu_{n_m}}-g(\nu_{n_m}))
  -h(g(n_m))\right){\rm 1}\{\widehat  T_g>\nu_{n_m}\},
\end{align*}
where we used the subadditivity of $h$ in the latter inequality.  
Now since $S_{\nu_n}\ge c(\varepsilon_n n)$ we have, 
\begin{align*}
  &\e\left[h(S_{n_m}-g(n_m))  {\rm 1}\{\widehat  T_g>n_m\} \mid \mathcal
    F_{\nu_{n_m}}\right]\\
 &\hspace{1cm}\ge h(S_{\nu_{n_m}}-g(\nu_{n_m})) \left(
    1-\frac{h(g(n_m))}{h(c(\varepsilon_{n_m} n_m))}\right){\rm
    1}\{\widehat  T_g>\nu_{n_m}\}.
\end{align*}
Hence, using the latter inequality and using 
the submartingale property we have, 
\begin{align*}
  E_1 &\ge \left(
    1-\frac{h(g(n_m))}{h(c(\varepsilon_{n_m} n_m))}  \right) 
  \e [h(S_{\nu_{n_m}}-g(\nu_{n_m}));\widehat
  T_g>\nu_{n_m},\nu_{n_m}\le n_{m-1}]\\
&\ge 
\left(
    1-\frac{h(g(n_m))}{h(c(\varepsilon_{n_m} n_m))}  \right) 
 \e [h(S_{n_{m-1}}-g(n_{m-1}));\widehat  T_g> n_{m-1},\nu_{n_m}\le
 n_{m-1}]\\
&\ge 
\left(
    1-\frac{h(g(n_m))}{h(c(\varepsilon_{n_m} n_m))}  \right) 
 \e [h(S_{n_{m-1}}-g(n_{m-1}));\widehat  T_g> n_{m-1}] \\
&\hspace{0.5cm} -
\e [h(S_{n_{m-1}}-g(n_{m-1}));\widehat  T_g> n_{m-1},\nu_{n_m}>
 n_{m-1}].
\end{align*}
Now note 
\begin{align*}
E_2&:=\e [h(S_{n_{m-1}})-g(n_{m-1})));\widehat  T_g> n_{m-1},\nu_{n_m}>n_{m-1} ]\\
&\le  \e\left[h(c(\varepsilon_{n_m} n_m));\nu_{n_m}> n_{m-1}\right]\\
&\le \left(h(c(\varepsilon_{n_m} n_m))\right)\pr(\nu_{n_m}> n_{m-1}).
\end{align*}
Applying Lemma \ref{lem1} with $\varepsilon_n:=(3\delta C \log n)^{-1}$, we obtain
\begin{align*}
E_2\leq (h(c(\varepsilon_{n_m} n_m))n_m^{-3}.
\end{align*}

As a result we have 
\begin{multline*}
\e [h(S_{n_m}-g(n_m));\widehat T_g>n_m]\\
\ge \left(1-\frac{h(g(n_m))}{h(c(\varepsilon_{n_m} n_m))}\right) 
 \e [h(S_{n_{m-1}}-g(n_{m-1}));\widehat T_g> n_{m-1}]\\
-\left(h(c(\varepsilon_{n_{m}} n_{m}))\right) / n_m^3.
\end{multline*}
Iterating this procedure $m$ times we obtain 
\begin{align*}
  &\e [h(S_{n_m}-g(n_m));\widehat T_g>n_m]  \ge 
\prod_{j=1}^m \left(
    1-\frac{h(g(n_j))}{h(c(\varepsilon_{n_j} n_j))}  \right) \\
&\hspace{0.5cm}\times
\left( \e [h(S_{n_0}+g(n_0));\widehat T_g>n_0]  - \sum_{j=1}^m\left(h(c(\varepsilon_{n_j} n_j))\right) / n_j^3\right).
\end{align*}
Now note \eqref{g-condition2} implies that
$$
\sum_{j=1}^\infty \frac{h(g(n_j))}{h(c(\varepsilon_{n_j} n_j))} <\infty. 
$$ 
Therefore, the product
$$
\prod_{j=1}^\infty \left(1-\frac{h(g(n_j))}{h(c(\varepsilon_{n_j} n_j))}  \right) 
$$
is finite. Furthermore, recalling that the index of $h$ is $\alpha(1-\rho)$, we conclude that
$$
\sum_{j=0}^\infty \frac{(h(c(\varepsilon_{n_j} n_j))}{n_j^3}<\infty.
$$

Then the statement of the proposition immediately follows.

\end{proof}

\subsection{Derivation of the asymptotics}

In this case the upper bound is obvious 
\begin{equation}
\label{hat.UpperBound}
\mathbf{P}(\widehat T_g>n)\le \mathbf{P}(T_0>n).
\end{equation}

According to \eqref{g-condition2},
$$
\int_x^{2x}\frac{h(g(y))}{yh(c(y/\log y))}dy\to0\quad\text{as }x\to\infty.
$$
Since functions $c(x)$ $h(x)$, $g(x)$ are increasing and $h$, $c$ are
in addition regularly varying, we conclude that
$$
\frac{h(g(x))}{h(c(x/\log x)}\to0.
$$
In other words,
$$
g(x)=o(c(x/\log x)).
$$
Consequently, there exists a sequence $\delta_n\to0$ such that
$$
g(n)=o(c(\delta_n n/\log n)).
$$
Moreover, we may asume that $\delta_n$ is such that slowly  varying
functions in $\pr(\tau>n)\sim n^{\rho-1}L(n)$ and 
$h(x)=x^{\alpha(1-\rho)}l(x)$ are such that $l(\delta_n n)\sim l(n)$ 
and $L(\delta_n n )\sim L(n)$.

Applying Lemma \ref{lem.concentration} with $\delta=\delta_n$ and
$\varepsilon_n=\delta_n/(2C\log n)$, we get
\begin{equation}
\label{hat.1.split}
\pr(\widehat T_g>n) = \pr(\widehat T_g>n,\nu_n <n\delta_n) +o(n^{-2}). 
\end{equation}

Let $A_n\uparrow \infty$ be such that $A_nc(\delta_n n/\log n)=o(c(n)).$ 
Set, for brevity, $r_n:=c(\delta_n n/\log n)$. We split the probability
$\pr(\widehat T_g>n,\nu_n\leq \delta_n n)$ into two parts:
\begin{align}
\label{hat.split}
\nonumber
  \pr(\widehat T_g>n,\nu_n\leq \delta_n n) 
= &\pr\left(\widehat T_g>n,\nu_n <\delta_n n, S_{\nu_n}\in (r_n, A_n r_n)\right) \\
&+\pr(\widehat T_g>n,\nu_n <\delta_n n, S_{\nu_n}> A_n r_n).
\end{align}
For the first term we are gong to apply \eqref{lem2.eq1}.  We have, 
\begin{align}
\label{hat.up.bound}
\nonumber
&\pr\left(\widehat T_g>n,\nu_n <\delta_n n, S_{\nu_n}\in (r_n, A_nr_n)\right)\\
\nonumber
&\hspace{1cm}\le \int_{r_n}^{A_n r_n} \pr (\widehat T_g>\nu_n, S_{\nu_n}
\in dy,\nu_n\le n\delta_n )\pr(\widetilde \tau_{y}>n-n\delta_n)\\
\nonumber
&\hspace{1cm}\sim\int_{r_n}^{A_n r_n}\pr(\widehat T_g>\nu_n, S_{\nu_n}\in dy,\nu_n\le n\delta_n )
h(y)\pr\{T_0>n\}\\
\nonumber
&\hspace{1cm}\le \e [h(S_{\nu_n});\widehat T_g>\nu_n] \pr\{T_0>n\}\\
&\hspace{1cm}\sim \widehat{V}(g)\pr\{T_0>n\}.
\end{align}
The bound
from below can be obtained by similar arguments:
\begin{align}
\label{hat.l.bound}
\nonumber
&\pr\left(\widehat T_g>n,\nu_n <\delta_n n, S_{\nu_n}\in (r_n, A_nr_n)\right)\\
\nonumber
&\ge \int_{r_n}^{A_n r_n} \pr (\widehat T_g>\nu_n, S_{\nu_n}
\in dy,\nu_n\le n\delta_n )\pr(\widetilde \tau_{y-g(n)}>n-n\delta_n)\\
\nonumber
&\sim\int_{r_n}^{A_n r_n}\pr(\widehat T_g>\nu_n, S_{\nu_n}\in dy,\nu_n\le n\delta_n )
h(y)\pr(T_0>n)\\
&= \e [h(S_{\nu_n});\widehat T_g>\nu_n, S_{\nu_n}<A_nr_n, \nu_n\le n\delta_n] \pr(T_0>n).
\end{align}

Now we turn to the second term in \eqref{hat.split}.
Using the Markov property and \eqref{lem2.eq1}, we obtain
\begin{multline*}
  \pr(\widehat T_g>n,\nu_n <\delta_n n, S_{\nu_n}> A_n r_n,\widehat T_g>\nu_n) \\
\le C \e [h(S_{\nu_n});S_{\nu_n}> A_n r_n,\nu_n <\delta_n n,\widehat T_g>\nu_n] \pr(T_0>n).
\end{multline*}
So we are left to prove that 
\begin{equation}
\label{hat.tail.expectation}
\e [h(S_{\nu_n});S_{\nu_n}> A_n r_n,\nu_n <\delta_n n,\widehat T_g>\nu_n] \to 0. 
\end{equation}
But this immediately follows from \eqref{tail.expectation} since 
$\widehat T_g\le  T_g$.  

Noting that \eqref{hat.tail.expectation} yields 
$$
\pr(\widehat T_g>n,\nu_n <\delta_n n, S_{\nu_n}> A_n r_n,\widehat T_g>\nu_n)
=o(\pr(T_0>n))
$$
and taking into account \eqref{hat.up.bound}, we get
\begin{equation}
\label{hat.up.bound2}
 \limsup_{n\to\infty}\frac{\mathbf{P}(\widehat T_g>n)}{\mathbf{P}(T_0>n)}\leq \widehat{V}(g).
\end{equation}

Combining \eqref{hat.l.bound} and \eqref{hat.tail.expectation}, we have
\begin{equation}
\label{hat.l.bound2}
 \liminf_{n\to\infty}\frac{\mathbf{P}(\widehat T_g>n)}{\mathbf{P}(T_0>n)}\geq 
\liminf_{n\to\infty}\e [h(S_{\nu_n});\widehat T_g>\nu_n,\nu_n\le n\delta_n].
\end{equation}

It follows from Lemma \ref{lem.concentration} with $\delta=\delta_n$ and
$\varepsilon_n=\delta_n/(2C\log n)$ that
$$
\e[h(S_{\delta_n n}-g(\delta_n n));\widehat T_g>\delta_n n,\nu_n> n\delta_n]\leq 
g(c(n))\mathbf{P}(\nu_n> n\delta_n)\to 0.
$$
Consequently,
$$
\e [h(S_{\nu_n});\widehat T_g>\nu_n,\nu_n\le n\delta_n]=
\e [h(S_{\theta_n}-g(\theta_n));\widehat T_g>\theta_n]+o(1),
$$
where $\theta_n:=\nu_n\wedge \delta_nn$.
Applying the optional stopping theorem to the supermartingale
$h(S_n-g(n)){\rm 1}\{\widehat T_g>n\}$, we get,
$$
\e [h(S_{\theta_n}-g(\theta_n));\widehat T_g>\theta_n]\geq \widehat V(g).
$$ 
Combining this with \eqref{hat.up.bound2}, we complete the proof of Theorem \ref{T3}.

\end{document}